\documentclass[12pt]{amsart}

\usepackage{amssymb}
\usepackage{amsmath}
\usepackage{mathdots}
\usepackage{amsbsy}
\usepackage{amscd}
\usepackage{amsthm}
\usepackage{mathrsfs}
\usepackage{verbatim}
\usepackage[colorlinks]{hyperref}
\usepackage[english]{babel}
\usepackage{url}
\usepackage{amsaddr}

\usepackage[english]{babel}
\usepackage[T1]{fontenc}
\usepackage[utf8]{inputenc}
\usepackage{tikz}

\usepackage[top=1.35in, bottom=1.5in, left=1.5in, right=1.5in]{geometry}

\theoremstyle{definition}
\newtheorem{Theorem}{Theorem}[section]
\newtheorem{Lemma}[Theorem]{Lemma}
\newtheorem{Example}[Theorem]{Example}
\newtheorem{Remark}[Theorem]{Remark}

\newtheorem{Corollary}[Theorem]{Corollary}
\newtheorem{Proposition}[Theorem]{Proposition}
\newtheorem{Fact}[Theorem]{Fact}

\newcommand{\C}{\mathbb{C}}

\newcommand{\R}{\mathbb{R}}

\newcommand{\Z}{\mathbb{Z}}

\newcommand{\ZZ}{\mathsf{Z}}

\newcommand{\mc}[1]{\mathcal{#1}} 
\newcommand{\mt}[1]{\text{#1}}

\newcommand{\msf}[1]{\mathsf{#1}}

\DeclareMathOperator{\Ima}{Im}

\newcommand{\aff}[1]{#1_{\text{aff}}}
\newcommand{\ant}[1]{#1_{\text{ant}}}

\begin{document}

\title{Toroidal affine Nash groups}

\author{Mahir Bilen Can}

\address{Tulane University, New Orleans}

\email{mcan@tulane.edu}

\date{September 29, 2015}


\maketitle

\begin{abstract}
A toroidal affine Nash group is the affine Nash group analogue of an anti-affine algebraic group. 
In this note, we prove analogues of Rosenlicht's structure and decomposition theorems: 
(1) Every affine Nash group $G$ has a smallest normal affine Nash subgroup $H$ such that 
$G/H$ is an almost linear affine Nash group, and this $H$ is toroidal. 
(2) If $G$ is a connected affine Nash group, then there exist a largest toroidal affine Nash subgroup $\ant{G}$ and 
a largest connected, normal, almost linear affine Nash subgroup $\aff{G}$. 
Moreover, we have $G=\ant{G}\aff{G}$, and $\ant{G}\cap \aff{G}$ contains $\aff{(\ant{G})}$ as an affine Nash subgroup of finite index.
\end{abstract}

\section{Introduction}

In this manuscript we are concerned with the structure theory of affine Nash groups. 
Building blocks of these objects are semi-algebraic sets, which, by definition, are finite unions of finite intersections of the sets of the form 
$$
\{x\in \R^n:\ f(x)=0, g_1(x) >0, g_2(x) > 0, \dots, g_k(x) > 0 \},
$$ 
where $f,g_1,\dots, g_k$ are polynomials on $\R^n$.
An affine Nash manifold in $\R^n$ is a semi-algebraic analytic submanifold of $\R^n$, 
and an affine Nash group is an affine Nash manifold in $\R^n$ with a compatible group structure.
In the category of affine Nash manifolds, a morphism between two objects is an analytic mapping with a semi-algebraic graph.
Such functions are called Nash mappings.
For a detailed survey of Nash category and related objects, we recommend Shiota's~\cite{Shiota1,Shiota2}.

Thanks to Hrushovski and Pillay, we know that affine Nash groups are precisely the finite covers of 
real algebraic groups (see~\cite{HP1, HP2}).
It is a paradoxical phenomenon in realm of Nash groups that even though they are defined over ordered 
fields, hence their geometry is more lucid compared to algebraic groups, 
reasonable geometric notions in Nash groups need to be defined more abstractly. 
Precisely this counter intuitiveness hinders the speed of discoveries in this field. 
For example, in literature the appropriate analogue of a linear algebraic group among affine Nash groups has only been considered recently in~\cite{Sun} by Sun, 
who coined the name {\em almost linear affine Nash group}.
These are Nash groups which admit a real, finite dimensional Nash representation with a finite kernel.

An important structure theorem in algebraic group theory, which is originally due to Chevalley states that 
a connected algebraic group $\msf{G}$ over a perfect field $k$ contains a unique closed connected normal affine algebraic 
subgroup $\msf{H}$ for which the quotient group $\msf{G}/\msf{H}$ is an abelian variety. 
A modern proof of this fundamental result is given by Conrad in~\cite{Conrad}.

In a recent manuscript,~\cite{FangSun}, Fang and Sun show how to generalize 
Chevalley's theorem in Nash context by defining ingeniously what the correct analogue of an abelian variety should be; 
an affine Nash group is said to be {\em complete} if it has no non-trivial connected almost linear affine Nash subgroup. 
An {\em abelian Nash manifold} is a connected, complete affine Nash group.
As an instructive non-example, consider $\mt{SO}(2,\R)$, the special orthogonal group of $2\times 2$ matrices with entries from $\R$.
Certainly, $\mt{SO}(2,\R)$ is a commutative, connected linear algebraic group, hence it is an almost linear Nash group. 
The point is that a commutative affine Nash group is not necessarily an abelian Nash manifold.
(A basic example of abelian Nash groups is given in the sequel.)

\begin{Theorem}[\cite{FangSun}, Theorem 1.2.]\label{T:Chevalley}
Let $G$ be a connected affine Nash group. 
Then there exists a unique connected, normal, almost linear affine Nash subgroup $H$ of $G$ such that $G/H$ is an abelian Nash manifold.
\end{Theorem}

It should be noted that Chevalley's structure theorem does not have
a direct generalization to complex Lie groups and the complexification of a Nash group is a complex Lie group. 
Therefore, Theorem~\ref{T:Chevalley} is one of the best structure results in this regard.
\begin{Remark}[Michel Brion]
A natural extension of Chevalley's theorem for complex Lie groups would assert that every complex connected Lie group is
an extension of a complex torus by a complex connected Stein Lie group, in a unique way. But this
fails in the following example: Consider an elliptic curve $E$ and its universal vector extension $G$. 
This is an extension of $E$ by the additive group of $\C$, as in Chevalley's theorem. 
But as a complex Lie group, $G$ is isomorphic to $\C^* \times \C^*$ and hence is Stein (see~\cite{Neeman}).
\end{Remark}

The main goal of this note is to prove another structure theorem for affine Nash groups,
whose origin goes back to the paper~\cite{Rosenlicht} of Rosenlicht.
An algebraic group $\msf{H}$ defined over a field $k$ is called {\em anti-affine}
if every global regular function on $\msf{H}$ is constant. Thus, if we denote by $\mc{O}(\msf{H})$ the ring of global regular functions on $\msf{H}$,
then $\msf{H}$ is anti-affine if and only if $\mc{O}(\msf{H})=k$. 
\begin{Theorem}(\cite{Brion15}, Theorem 1)\label{T:Brion15}
Every algebraic group over a field $k$ has a smallest normal subgroup scheme $\msf{H}$ such that 
the quotient $\msf{G}/\msf{H}$ is affine. In this case, $\msf{H}$ is smooth, connected, and lies in the center of $\msf{G}^0$;
in particular $\msf{H}$ is commutative. Moreover, $\mc{O}(\msf{H})=k$ and $\msf{H}$ is the largest group scheme of 
$\msf{G}$ satisfying this property. The formation of $\msf{H}$ commutes with field extensions.  
\end{Theorem}

There is a natural extension of the notion of anti-affineness to complex Lie groups; they are called {\em Cousin groups}, or more
suggestively as {\em toroidal groups}. A complex Lie group $G$ is said to be toroidal if every global holomorphic 
function on $G$ is constant. A toroidal group admits an explicit description as a $(\C^*)^m$-fiber bundle over a torus,
hence the nomenclature follows. For more on these objects we recommend~\cite{AbeKopfermann}.
Although these groups are at the center of investigation for more than a century (starting with P. Cousin's work from 1910), 
to the best of our knowledge, in the literature there is no analogue of Rosenlicht's result for complex Lie groups.

Motivated by the results that are mentioned above, we define toroidal affine Nash groups as follows.
An affine Nash group is said to be anti-linear if it has no quotient Nash group which is almost linear and positive dimensional. 
Then a {\em toroidal affine Nash group} is defined to be a connected anti-linear affine Nash group.  
(We thank Bingyong Sun for proposing this definition of toroidal affine Nash groups.)
Examples of toroidal affine Nash groups include abelian Nash groups. 
\begin{Example}
Let $\alpha$ be a nonzero real number. There exists unique Nash group structure on $(\R/\Z)_{\alpha}:=\R/\Z$
such that the Weierstrass elliptic function $\wp=\wp_\alpha$ defined by 
$$
\wp (x) := \frac{1}{x^2} + \sum_{\omega \in \Z + \alpha \sqrt{-1}\Z,\ \omega \neq 0} \left( \frac{1}{(x-\omega)^2} - \frac{1}{\omega^2} \right)
$$ 
is a Nash function on $(\R/\Z)_{\alpha} - \{ 0 \}$.
See~\cite{MS}, Section 4. It is shown in~\cite{FangSun} that $(\R/\Z)_{\alpha}$ is an abelian Nash group, hence it is a toroidal affine Nash group. 
The basic idea of the proof is to pass to an algebraization (a concept that we review in the sequel) of $(\R/\Z)_{\alpha}$ which results in an elliptic curve. 
\end{Example}

Our first main result related to toroidal affine Nash groups is the following:

\begin{Theorem}\label{T:Rosenlicht}
Every affine Nash group $G$ has a smallest normal affine Nash subgroup $H$ such that 
$G/H$ is an almost linear affine Nash group, and this $H$ is toroidal. 
Moreover, $H$ lies in the center of the identity connected component $G^0$ of $G$; in particular $H$ is commutative. 
\end{Theorem}
We prove this theorem in the next two sections by a sequence of lemmas.

Let $\msf{G}$ be a connected algebraic group defined over a perfect field $k$. 
By Theorem~\ref{T:Brion15} we know that $\msf{G}$ has a smallest normal subgroup $\ant{\msf{G}}$ such that $\msf{G}/\ant{\msf{G}}$ is an affine algebraic group and $\ant{\msf{G}}$ is anti-affine. 
It turns out $\ant{\msf{G}}$ is the kernel of the natural morphism $\msf{G}\rightarrow \mt{Spec}\ \mc{O}(\msf{G})$. See Section 3.2~\cite{Brion15}. 
In a sense orthogonal to $\ant{\msf{G}}$ there exists a smallest connected affine subgroup $\aff{\msf{G}}\subseteq \msf{G}$ such that 
$\msf{G}/\aff{\msf{G}}$ is an abelian variety (whose existence is guaranteed by Chevalley's theorem). 
Thus, one has the {\em Rosenlicht decomposition}: $\msf{G}=\aff{\msf{G}}\ant{\msf{G}}$ and $\ant{\msf{G}}\cap \aff{\msf{G}}$ contains 
$\aff{(\ant{\msf{G}})}$ as an algebraic subgroup of finite index. 
For a modern proof of this fundamental result and more (including another proof of Chevalley's structure theorem) we recommend the excellent expos\'e~\cite{Brion15} by Brion. 
Our second main result is an analogous decomposition for affine Nash groups.

\begin{Theorem}\label{T:Rosenlicht decomposition}
If $G$ is a connected affine Nash group, then there exist a largest toroidal affine Nash subgroup $\ant{G}$ and 
a largest connected almost linear affine Nash subgroup $\aff{G}$. 
Moreover, the following hold true
\begin{itemize}
\item $G=\ant{G}\aff{G}$;
\item $\ant{G}\cap \aff{G}$ contains $\aff{(\ant{G})}$ as an affine Nash subgroup of finite index.
\end{itemize}
\end{Theorem}

Notice that the group $H$ of Theorem~\ref{T:Rosenlicht} equals the group $\ant{G}$ of Theorem~\ref{T:Rosenlicht decomposition} 
and $H$ of Theorem~\ref{T:Chevalley} equals the group $\aff{G}$ of Theorem~\ref{T:Rosenlicht decomposition}.
We prove this theorem in the last section of our paper.

\vspace{.5cm}

\textbf{Acknowledgements.} We thank Professor Michel Brion for his generous comments and feedback 
on earlier versions of this paper. We thank Professor Binyong Sun for his extremely valuable feedback
and for his suggestion of the correct definition of toroidal affine Nash groups.
Finally, we thank the anonymous referee for her/his careful reading and great suggestions which improved the quality of the paper drastically. 

We acknowledge that the initial seeds of this work were laid during 
{\em Clifford Lectures on Algebraic Groups: Structure and Actions}, which was supported by an 
NSF Workshop Grant (DMS-1522969).

\section{Notation and Preliminaries}

The purpose of this section is to set our notation forth and recall from literature some basic information about 
Nash manifolds. We recommend~\cite{Shiota1},~\cite{Shiota2} and~\cite{BCR}.

Recall that a semi-algebraic set in $\R^n$ is a finite union of finite intersection of the 
sets of the form $\{x\in \R^n:\ p(x) = 0\}$ and $\{y\in \R^n:\ q(y) > 0 \}$, where $p$ and $q$ 
are polynomial functions on $\R^n$. An analytic function $f: U\subseteq \R^n \rightarrow \R$ defined on an open semi-algebraic subset $U$  
is said to be {\em Nash} if its graph in $\R^{n+1}$ is contained in the vanishing set of a non-zero polynomial in $n+1$-variables. In other words, $p(x,f(x))=0$ for some $p(x_1,\dots,x_{n+1})\in \R[x_1,\dots, x_{n+1}]\setminus \{0\}$.

Following~\cite{Shiota2}, we define an affine Nash manifold as follows: A {\em Nash set} $X$ in an open semi-algebraic set $U$ 
is a common zero set of Nash functions on $U$, and a {\em Nash function on $X$} is the restriction of a Nash function from $U$. 
We denote by $\mc{N}(X)$ the ring of global Nash functions on $X$. 
A Nash set $X$ is said to be an affine Nash manifold if for each $x\in X$ the 
local ring of $\mc{N}(X)$ at $x$ is a regular local ring of constant Krull dimension. 
This (local) definition indicates a certain similarity of Nash manifolds to smooth real algebraic varieties.
Indeed, any real affine nonsingular algebraic variety in $\R^n$ is an affine Nash manifold in $\R^n$, and 
any nonsingular real algebraic variety has a Nash manifold structure (see Remark I.3.10 in~\cite{Shiota1}).
Conversely, a Nash manifold is a compact or compactifiable as an analytic manifold (see Theorem IV.1.1 in~\cite{Shiota1}),
and we know from Theorem B in~\cite{AK85} that 
if $X$ is a compact smooth submanifold of $\R^n$, then $X$ is $\varepsilon$-isotopic to a nonsingular real algebraic subset of $\R^{n+1}$.
Here, $\varepsilon$-isotopic means that there are arbitrarily small smooth isotopies to components of real algebraic varieties.

Recall that an affine Nash group is an affine Nash manifold with Nash group operations.
More generally, one defines an (abstract) Nash manifold in terms of Nash coordinate charts, and then defines a Nash group
as an abstract Nash manifold with a compatible group structure. Since we do not need this generality in our paper
we skip its notation. However, we should mention that a complete classification of all connected one-dimensional (not necessarily affine) Nash groups is known. 
This classification is the content of the paper~\cite{MS}.

Affine Nash groups are precisely the finite covers of open subgroups of the real point groups of real algebraic groups,
see \cite[Theorem 2.1]{HP2}. 
However, this does not mean that an affine Nash group is necessarily a real algebraic group. 
Consider, for example, $\mt{SL}(2,\R)$, the special linear group of $2\times 2$ matrices over real numbers. 
It is well-known that even the double-cover of $\mt{SL}(2,\R)$ (so called metaplectic group) has no finite dimensional faithful linear representation.

To distinguish between Nash groups and algebraic groups throughout the text, 
we use ordinary fonts (in capitals) for the former and use sans-serif fonts (in capitals) for the latter. 
If $X$ is an algebraic variety or a scheme, then its ring of regular functions is denoted by $\mc{O}(X)$. 
In particular, if $\msf{G}$ is a real algebraic group, then $\mc{N}(\msf{G}(\R))$ is an algebra over $\mc{O}(\msf{G})$ via the natural homomorphism 
$\mc{O}(\msf{G})\rightarrow \mc{N}(\msf{G}(\R))$.

It is proven in~\cite[Proposition 4.1]{FangSun} that the quotients of affine Nash groups by their Nash subgroups are 
naturally affine Nash manifolds, and when the Nash subgroups are normal, the quotients are naturally affine Nash groups.
This result is obtained by making use of a powerful tool called ``algebraization,'' which is introduced in the same paper. 
An {\em algebraization of a Nash group} $G$ is an algebraic group $\msf{G}$ defined over $\R$ 
together with a Nash group homomorphism $G\rightarrow \msf{G}(\R)$ which has a finite kernel and whose image is 
Zariski dense in $\msf{G}$. 

We use the following facts from~\cite{FangSun} later in the text.
\begin{Fact}\label{Facts}
Let $G$ be a Nash group. 
\begin{enumerate}
\item $G$ has an algebraization $\msf{G}$ if and only if $G$ is an affine Nash group (see \cite[Theorem 3.3]{FangSun}).
\item Let $\msf{G}$ be any algebraization of $G$. Then $G$ is complete if and only if $\msf{G}$ is complete as an algebraic variety (see \cite[Lemma 3.8]{FangSun}) and $G$ is almost linear if and only if $\msf{G}$ is affine as an algebraic variety.
\item If $G$ is a connected complete affine Nash group, then it is compact and commutative as a Lie group 
(see \cite[Corollary 3.9]{FangSun}).
\end{enumerate}
\end{Fact}

Although we do not utilize the following fact, it is interesting to observe it:
\begin{Remark}\label{L:universal}
Let $G$ be an affine Nash group. 
Let $\phi : G \rightarrow \msf{G}$ be an algebraization and let $Z$ and $\ZZ$ denote 
the centers of $G$ and $\msf{G}$, respectively. In this case, $\phi(Z)\subseteq \ZZ$. 
Indeed, the image of $\phi$ is dense in $\msf{G}$, therefore, 
the inner automorphism of $\msf{G}$ defined by an element of $\phi(Z)$ is identity on a Zariski dense subset.
\end{Remark}

We end our preparatory section by another useful remark.

\begin{Remark}
It is a consequence of a remarkable result of Pillay~\cite{Pillay} that 
every semialgebraic subgroup of a Nash group $G$ is a closed Nash 
submanifold of $G$. In particular, it follows that the image of a Nash homomorphism 
$\phi: G\rightarrow G'$ is a Nash subgroup of $G$. In particular, it is closed in $G'$.
We use this fact in the sequel without mentioning it.
\end{Remark}

\section{Proof of Theorem~\ref{T:Rosenlicht}}

Recall that an affine Nash group $G$ is called toroidal if it is connected and have no positive dimensional almost linear quotient. 
In particular, the trivial group is toroidal. 
The proofs of the next three results are straightforward so we skip them.

\begin{Proposition}\label{P:image and quotient}
If $f: G \rightarrow G'$ is a Nash group homomorphism from a toroidal affine Nash group $G$, 
then both of $\Ima f$ and $G/ \ker f$ are toroidal affine Nash groups as well.
\end{Proposition}

\begin{Lemma}\label{L:no maps}
Let $M$ be a toroidal affine Nash group, and let $f: M \rightarrow M'$ be a Nash homomorphism into an almost linear affine Nash group $M'$.
Then $f$ is constant.
\end{Lemma}

\begin{Lemma}\label{L:toroidal->complete}\label{L:image of an anti-affine}
Let $G$ be a toroidal affine Nash group. If $\phi: G \rightarrow \msf{G}$ is an algebraization, 
then $\msf{G}$ is an anti-affine algebraic group.
\end{Lemma}

\begin{Corollary}\label{C:central}
Let $H$ be a toroidal affine Nash subgroup of a connected affine Nash group $G$.
Then $H$ is central; $H\subseteq Z(G)$. 
\end{Corollary}

\begin{proof}
Apply Lemma~\ref{L:image of an anti-affine} to the inclusion homomorphism $\iota: H\hookrightarrow G$.
The result now follows from the corresponding result for algebraic groups.
\end{proof}

Next, we obtain a useful criterion for toroidalness via algebraization. 
\begin{Lemma}\label{L:converse}
Let $G$ be a connected affine Nash group and $\phi: G \rightarrow \msf{G}$ be an algebraization. 
If $\msf{G}$ is an anti-affine algebraic group, then $G$ is a toroidal affine Nash group.
\end{Lemma}
\begin{proof}

Towards a contradiction under our hypothesis we assume that $G$ is not toroidal, hence it has a proper normal 
subgroup $H\lhd G$ such that $G/H$ is an almost linear Nash group. Let $\pi : G\rightarrow G/H$ denote 
the natural quotient homomorphism. 
By Lemma 3.7 of~\cite{FangSun},
there exist an algebraization $\phi' : G/H \rightarrow \msf{G}'$ and a morphism $\tau : \msf{G}\rightarrow \msf{G}'$ 
such that the following diagram commutes: 
\begin{figure}[htp]\label{F:commutative}
\begin{center}
\begin{tikzpicture}
\node at (-2.25,1) (a) {$G$};
\node at (2.25,1) (b) {$G/H$};
\node at (-2.25,-1) (c) {$\msf{G}$};
\node at (2.25,-1) (d) {$\msf{G}'$}; 
\node at (-2.55,0) {$\phi$}; 
\node at (2.55,0) {$\phi'$}; 
\node at (0,1.35) {}; 
\node at (0,-1.35) {$\tau$}; 
\draw[->,thick] (a) to (b);
\draw[->,thick] (a) to (c);
\draw[->,thick] (b) to (d);
\draw[->,thick] (c) to (d);
\end{tikzpicture}
\end{center}
\end{figure}

Since $G/H$ is almost linear, $\msf{G}'$ is affine. However, $\msf{G}$ is anti-affine and connected, therefore, $\tau$ is trivial. 
It follows that $H=G$. This contradiction shows that $G$ is toroidal. 
\end{proof}

Finally, we are ready to prove our first main theorem.

\begin{proof}[Proof of Theorem~\ref{T:Rosenlicht}]
Let $G$ be an affine Nash group, $\phi : G \rightarrow \msf{G}$ be an algebraization of $G$, 
and let $\ant{\msf{G}}\subset \msf{G}$ be the smallest normal subgroup of $\msf{G}$ such that $\msf{G}/\ant{\msf{G}}$ is affine. 
By Rosenlicht's theorem we know that $\ant{\msf{G}}$ is anti-affine. 
Let $H\subseteq G$ denote the identity component of the preimage of $\ant{\msf{G}}$ in $G$, so 
$H$ is toroidal by Lemma~\ref{L:converse}.
We claim that $H$ is the smallest normal affine Nash group of $G$ such that $G/H$ is an almost linear affine Nash group.

Since $\ant{\msf{G}}$ is normal in $\msf{G}$, its preimage $\phi^{-1}(\ant{\msf{G}})$ is normal in $G$, hence, so is $H$. 
Moreover, by Lemma 3.5~\cite{FangSun} we get an induced algebraization map 
$\tilde{\phi} : G/H \rightarrow G/\phi^{-1}(\ant{\msf{G}}) \rightarrow \msf{G} / \ant{\msf{G}}$, 
hence $G/H$ is almost linear. Thus, it remains to show that $H$ is smallest with respect to this property.

Let $K\subset G$ be another normal affine Nash subgroup of $G$ such that $G/K$ is an almost linear affine Nash group.
On one hand, $H/K\cap H$ embeds into $G/K$, hence it is an almost linear affine Nash group.
On the other hand, $H/K\cap H$ is toroidal as a quotient of $H$. 
Passing to an algebraization $\psi: H/K\cap H \rightarrow \msf{S}$, we see that $\msf{S}$ is both affine and anti-affine
by Lemma~\ref{L:converse}. Hence, it is trivial. Since $H/ H \cap K$ is connected and 
since its image is trivial under a finite map, $H/H\cap K$ is trivial also. Therefore, $H \subset K$. The proof is complete.

\end{proof}

\section{Proof of Theorem~\ref{T:Rosenlicht decomposition}}




\begin{proof}[Proof of Theorem~\ref{T:Rosenlicht decomposition}]

Let $G$ be a connected affine Nash group and $\phi : G\rightarrow \mathsf G$ be an algebraization.
By Theorem~\ref{T:Rosenlicht} we know the existence of a smallest normal affine Nash subgroup $H$ of $G$ such that $G/H$ is almost linear.
We denote this $H$ by $G_{\mathrm{ant}}$.

Let $\aff{G}$ denote the identity connected component of the preimage of $\aff{\msf{G}}$ under $\phi$.
Here, $\aff{\msf{G}}$ is the largest connected normal affine subgroup of $\msf{G}$ whose existence is guaranteed by 
the Chevalley's structure theorem (Theorem~\ref{T:Chevalley}). 
Since $\aff{\msf{G}}$ is affine and normal in $\msf{G}$, the identity connected component $\aff{G}$ of $\phi^{-1}(\aff{\msf{G}}(\R))$ 
is normal and almost linear as well (for the latter we use Fact~\ref{Facts} (3)).

Morover, the  map $G/G_{\mathrm{aff}}\rightarrow \mathsf G/ \mathsf G_{\mathrm{aff}}$ induced by $\phi$ is an algebraization of $G/G_{\mathrm{aff}}$, and the  map $G/G_{\mathrm{ant}}\rightarrow \mathsf G/ \mathsf G_{\mathrm{ant}}$ induced by $\phi$ is an algebraization of $G/G_{\mathrm{ant}}$.Thus $G/G_{\mathrm{aff}}$ is an abelian Nash manifold and $G/G_{\mathrm{ant}}$ is an almost linear Nash group.

Let $K$ be a connected Nash subgroup of $G$. If $K$ is almost linear, then the natural homomorphism $K\rightarrow G/G_{\mathrm{aff}}$ is trivial, and thus $K\subset G_{\mathrm{aff}}$. This shows that $G_{\mathrm{aff}}$ is the largest connected almost linear Nash subgroup of $G$. Similarly, if $K$ is toroidal, then the natural homomorphism $K\rightarrow G/G_{\mathrm{ant}}$ is trivial, and thus $K\subset G_{\mathrm{ant}}$. This shows that $G_{\mathrm{ant}}$ is the largest toroidal Nash subgroup of $G$.

Since both $\mathsf{G}_{\mathrm{aff}}$ and $\mathsf{G}_{\mathrm{ant}}$ are normal in $\mathsf G$, both $G_{\mathrm{aff}}$ and $G_{\mathrm{ant}}$ are normal in $G$.
By comparing the Lie algebras, we know that $G=G_{\mathrm{aff}} G_{\mathrm{ant}}$.

Let $(G_{\mathrm{ant}})_{\mathrm{aff}}$ denote the  identity connected component of $\phi^{-1}((\mathsf{G}_{\mathrm{ant}})_{\mathrm{aff}}(\mathbb R))$. It is the largest connected almost linear Nash subgroup of $G_{\mathrm{ant}}$ by the above argument. It is clear that
$(G_{\mathrm{ant}})_{\mathrm{aff}}$ is a Nash subgroup of $G_{\mathrm{aff}}\cap G_{\mathrm{ant}}$. These two groups have the same Lie algebra since $(\mathsf{G}_{\mathrm{ant}})_{\mathrm{aff}}$ and $\mathsf{G}_{\mathrm{aff}}\cap\mathsf{G}_{\mathrm{ant}}$
 have the same Lie algebra.Thus $(G_{\mathrm{ant}})_{\mathrm{aff}}$ is a Nash subgroup of $G_{\mathrm{aff}}\cap G_{\mathrm{ant}}$ of finite index.

\end{proof}


\begin{thebibliography}{9}

\bibitem{AK85}
S. Akbulut, H. King, {\em 
On approximating submanifolds by algebraic sets and a solution to the Nash conjecture,} 
Invent. Math. 107 (1992), no. 1, 87--98.


\bibitem{AbeKopfermann} 
Y. Abe, K. Kopfermann, {\em Toroidal groups. Line bundles, cohomology and quasi-abelian varieties,} 
Lecture Notes in Mathematics, 1759. Springer-Verlag, Berlin, 2001.

\bibitem{BCR} J. Bochnak, M. Coste, M-F. Roy, 
{\em Real algebraic geometry,}
Translated from the 1987 French original. Ergebnisse der Mathematik und ihrer Grenzgebiete (3) 
[Results in Mathematics and Related Areas (3)], 36. Springer-Verlag, Berlin, 1998




\bibitem{Brion09} M. Brion, {\em Anti-affine algebraic groups,}
J. Algebra 321 (2009), no. 3, 934--952. 
\bibitem{Brion15} M. Brion, {\em Some structure theorems for algebraic groups,}
Proceedings of Symposia in Pure Mathematics, to appear. 
\url{http://arxiv.org/abs/1509.03059}




\bibitem{Conrad} B. Conrad, {\em A modern proof of Chevalley's theorem on algebraic groups,} 
J. Ramanujan Math. Soc. 17 (2002) 1--18.

\bibitem{DG} M. Demazure, P. Gabriel, {\em Introduction to algebraic geometry and algebraic groups,} 
Translated from the French by J. Bell. North-Holland Mathematics Studies, 39. 
North-Holland Publishing Co., Amsterdam-New York, 1980. 

\bibitem{FangSun} Y. Fang, B. Sun, {\em Chevalley's Theorem For Affine Nash Groups,}
Journal of Lie Theory Volume 26 (2015) 359--369.



\bibitem{HP1} E. Hrushovski, A. Pillay, 
{\em Groups definable in local fields and pseudofinite fields,} 
Israel J. Math. 85 (1994) 203--262.

\bibitem{HP2} E. Hrushovski, A. Pillay,
{\em Affine Nash groups over real closed fields.}
Confluentes Math. 3 (2011), no. 4, 577--585. 

\bibitem{MS} J. Madden, C.M. Stanton, {\em One-dimensional Nash groups,}
Pacific J. Math. 154 (1992), no. 2, 331--344. 


\bibitem{Mum} D. Mumford, {\em Abelian varieties,}
Tata Institute of Fundamental Research Studies in Mathematics, 5. 
Hindustan Book Agency, New Delhi, 2008.


\bibitem{Neeman} A. Neeman, {\em Steins, affines and Hilbert's fourteenth problem.}
Ann. of Math. (2) 127 (1988), no. 2, 229--244. 


\bibitem{Pillay} A. Pillay, {\em On groups and fields definable in o-minimal structures,} 
J. Pure Appl. Algebra 53 (1988), no. 3, 239--255.

\bibitem{Rosenlicht} M. Rosenlicht, {\em Some basic theorems on algebraic groups,} Amer. J. Math. 78 (1956) 401--443.

\bibitem{Shiota1} M. Shiota, {\em Nash manifolds,} 
Lecture Notes in Mathematics, 1269. Springer-Verlag, Berlin, 1987. 

\bibitem{Shiota2} M. Shiota, {\em Nash functions and manifolds,} 
Lectures in real geometry (Madrid, 1994), 69--112, de Gruyter Exp. Math., 23, de Gruyter, Berlin, 1996.


\bibitem{Sun} B. Sun, {\em Almost linear Nash groups,} Chin. Ann. Math. 36B(3), 2015, 355--400.


\end{thebibliography}
\end{document}